\documentclass[12pt]{amsart}

\usepackage{verbatim}
\usepackage{amssymb}
\usepackage{upref}
\usepackage[all]{xy}
\usepackage[usenames,dvipsnames]{color}
\usepackage{url}
\usepackage[normalem]{ulem}

\allowdisplaybreaks

\newcommand{\duality}{-crossed-product duality}
\newcommand{\satisfy}{satisfies $E$-\duality}

\newtheorem{thm}{Theorem}[section]
\newtheorem*{thm*}{Theorem}
\newtheorem{lem}[thm]{Lemma}
\newtheorem{cor}[thm]{Corollary}
\newtheorem{prop}[thm]{Proposition}
\newtheorem{obs}[thm]{Observation}
\newtheorem{conj}[thm]{Conjecture}

\theoremstyle{definition}
\newtheorem{defn}[thm]{Definition}
\newtheorem{q}[thm]{Question}

\newtheorem{ex}[thm]{Example}
\newtheorem{notn}[thm]{Notation}

\newtheorem*{notn*}{Notation}

\newtheorem*{hyp*}{Hypothesis}

\newtheorem{rem}[thm]{Remark}
\newtheorem*{rem*}{Remark}

\numberwithin{equation}{section}

\newcommand{\secref}[1]{Section~\textup{\ref{#1}}}

\newcommand{\thmref}[1]{Theorem~\textup{\ref{#1}}}
\newcommand{\corref}[1]{Corollary~\textup{\ref{#1}}}
\newcommand{\lemref}[1]{Lemma~\textup{\ref{#1}}}
\newcommand{\propref}[1]{Proposition~\textup{\ref{#1}}}
\newcommand{\obsref}[1]{Observation~\textup{\ref{#1}}}
\newcommand{\defnref}[1]{Definition~\textup{\ref{#1}}}
\newcommand{\remref}[1]{Remark~\textup{\ref{#1}}}

\newcommand{\exref}[1]{Example~\textup{\ref{#1}}}

\newcommand{\conjref}[1]{Conjecture~\textup{\ref{#1}}}

\newcommand{\midtext}[1]{\quad\text{#1}\quad}

\renewcommand{\)}{\textup)}

\newcommand{\Z}{\mathbb Z}

\newcommand{\C}{\mathbb C}

\newcommand{\KK}{\mathcal K}

\newcommand{\Chi}{\raisebox{2pt}{\ensuremath{\chi}}}
\renewcommand{\epsilon}{\varepsilon}

\DeclareMathOperator{\aut}{Aut}

\DeclareMathOperator{\supp}{supp}

\DeclareMathOperator*{\spn}{span}
\DeclareMathOperator*{\clspn}{\overline{\spn}}

\newcommand{\id}{\text{\textup{id}}}

\newcommand{\tr}{\text{\textup{tr}}}
\newcommand{\case}{& \text{if }}

\newcommand{\<}{\langle}
\renewcommand{\>}{\rangle}

\newcommand{\inv}{^{-1}}

\renewcommand{\bar}{\overline}
\newcommand{\what}{\widehat}

\newcommand{\csZg}{C^*_{C_0(G)}(G)}
\newcommand{\csLpg}{C^*_{L^p(G)}(G)}

\newcommand{\ann}{^\perp}

\newcommand{\ga}{group $C^*$-algebra}
\newcommand{\wkstcl}[1]{\bar{#1}^{\,\text{weak*}}}

\begin{document}
\title[Exotic group $C^*$-algebras]{Exotic group $C^*$-algebras in noncommutative duality}

\author[Kaliszewski]{S. Kaliszewski}
\address{School of Mathematical and Statistical Sciences
\\Arizona State University
\\Tempe, Arizona 85287}
\email{kaliszewski@asu.edu}

\author[Landstad]{Magnus~B. Landstad}
\address{Department of Mathematical Sciences\\
Norwegian University of Science and Technology\\
NO-7491 Trondheim, Norway}
\email{magnusla@math.ntnu.no}

\author[Quigg]{John Quigg}
\address{School of Mathematical and Statistical Sciences
\\Arizona State University
\\Tempe, Arizona 85287}
\email{quigg@asu.edu}



\date{August 30, 2012}

\begin{abstract}
We 
show that for
a locally compact group $G$
there is a one-to-one correspondence between $G$-invariant weak*-closed subspaces $E$ of the Fourier-Stieltjes algebra $B(G)$ containing $B_r(G)$ and 
quotients $C^*_E(G)$ of $C^*(G)$ which are intermediate between $C^*(G)$ and the reduced group algebra $C^*_r(G)$.
We show that the canonical comultiplication on $C^*(G)$ descends to a coaction or a comultiplication on $C^*_E(G)$ if and only if $E$ is an ideal or subalgebra, respectively.
When $\alpha$ is an action of $G$ on a $C^*$-algebra $B$,
we define ``$E$-crossed products'' $B\rtimes_{\alpha,E} G$ lying between the full crossed product and the reduced one,
and we conjecture that these ``intermediate crossed products'' satisfy an ``exotic'' version of crossed-product duality involving $C^*_E(G)$.
\end{abstract}

\maketitle

\section{Introduction}\label{intro}

It has long been known that for a locally compact group $G$ there are many $C^*$-algebras between the full group $C^*$-algebra $C^*(G)$ and the reduced algebra $C^*_r(G)$ (see \cite{eym}).
However, little study has been made regarding the extent to which these intermediate algebras can be called group $C^*$-algebras.

This paper is inspired by recent work of Brown and Guentner \cite{BrownGuentner}, 
which studies such intermediate algebras for discrete groups,
and \cite{Okayasu}, which shows that in fact there can be a continuum of such intermediate algebras.
We shall consider a general locally compact group $G$, and show that 
by elementary harmonic analysis
there is a one-to-one correspondence between $G$-invariant weak*-closed subspaces $E$ of the Fourier-Stieltjes algebra $B(G)$ containing $B_r(G)$ and 
quotients 
$C^*_E(G)$
of $C^*(G)$ which are intermediate between $C^*(G)$ and the reduced group algebra $C^*_r(G)$.

We are primarily interested in the following results:
\begin{itemize}
\item $E$ is an ideal if and only if there is a coaction $C^*_E(G)\to M(C^*_E(G)\otimes C^*(G))$.
\item $E$ is a subalgebra if and only if there is a comultiplication $C^*_E(G)\to M(C^*_E(G)\otimes C^*_E(G))$.
\end{itemize}
(See Propositions~\ref{coaction} and \ref{comultiplication} for more precise statements.)
These $C^*$-algebras can be used to describe various properties of $G$,
e.g., if $G$ is discrete and $E=\bar{B(G)\cap c_0(G)}$,
then $G$ has the Haagerup property if and only if $C^*_E(G)=C^*(G)$
(see \cite[Corollary~3.4]{BrownGuentner}).
Brown and Guentner also prove that (again, in the discrete case)
$C^*_E(G)$ is a compact quantum group, because it carries a comultiplication, and this caught our attention since it makes a connection with noncommutative crossed-product duality.

If we have a $C^*$-dynamical system $(B,G,\alpha)$,
one can form the full crossed product $B\rtimes_\alpha G$
or the reduced crossed product $B\rtimes_{\alpha,r} G$.
We show in \secref{exotic coaction} that for $E$ as above there is an ``$E$-crossed product'' $B\rtimes_{\alpha,E} G$,
and we speculate that these ``intermediate'' crossed products satisfy an ``exotic'' version of crossed-product duality involving $C^*_E(G)$.

After a short section on preliminaries, in \secref{certain quotients} we prove the above-mentioned results concerning the existence of a coaction or comultiplication on $C^*_E(G)$ .

In \secref{classical} we briefly explore the analogue for arbitrary locally compact groups of the construction used in \cite{BrownGuentner}, where for discrete groups they construct \ga s starting with ideals of $\ell^\infty(G)$.

In \secref{discrete} we specialize (for the only time in this paper) to the discrete case, showing that a quotient $C^*_E(G)$ is a \ga\ if and only if it is \emph{topologically graded} in the sense of \cite{ExelAmenability}.

Finally, in \secref{exotic coaction} we outline a possible application of our exotic group algebras to noncommutative crossed-product duality.

After this paper was circulated in preprint form, we learned that Buss and Echterhoff \cite{BusEch} have 
given counterexamples to \conjref{E-coaction}
and have proven \conjref{E dual}.

We thank the referee for helpful comments.

\section{Preliminaries}\label{prelim}

All ideals of $C^*$-algebras will be closed and two-sided.
If $A$ and $B$ are $C^*$-algebras, then $A\otimes B$ will denote the minimal tensor product.

For one of our examples we will need the following elementary fact, which is surely folklore.

\begin{lem}\label{onto}
Let $A$ be a $C^*$-algebra, and let $I$ and $J$ be ideals of $A$.
Let $\phi:A\to A/I$ and $\psi:A\to A/J$ be the quotient maps,
and define
\[
\pi=\phi\oplus \psi:A\to (A/I)\oplus (A/J).
\]
Then $\pi$ is surjective if and only if $A=I+J$.
\end{lem}

\begin{proof}
First assume that $\pi$ is surjective,
and let $a\in A$.
Choose $b\in A$ such that
\[
\pi(b)=\bigl(\phi(a),0\bigr),
\]
i.e., 
$\phi(b)=\phi(a)$ and $\psi(b)=0$.
Then $a-b\in I$,  $b\in J$, and $a=(a-b)+b$.

Conversely, assume that $A=I+J$,
and let $a\in A$.
Choose $b\in I$ and $c\in J$ such that $a=b+c$.
Then $\psi(c)=0$, and $\phi(c)=\phi(a)$ since $a-c\in I$.
Thus
\[
\pi(c)=\bigl(\phi(a),0\bigr).
\]
It follows that $\pi(A)\supset (A/I)\oplus \{0\}$, and similarly $\pi(A)\supset \{0\}\oplus (A/J)$,
and hence $\pi$ is onto.
\end{proof}

A point of notation: for a homomorphism between $C^*$-algebras, or for a bounded linear functional on a $C^*$-algebra, we use a bar to denote the unique strictly continuous extension to the multiplier algebra.

We adopt the conventions of \cite{enchilada} for actions and coactions of a locally compact group $G$ on a $C^*$-algebra $A$. In particular, we use \emph{full} coactions
$\delta:A\to M(A\otimes C^*(G))$,
which are nondegenerate injective homomorphisms satisfying the \emph{coaction-nondegeneracy} property
\begin{equation}\label{coaction nondegenerate}
\clspn\{\delta(A)(1\otimes C^*(G))=A\otimes C^*(G)
\end{equation}
and the \emph{coaction identity}
\begin{equation}\label{identity}
\bar{\delta\otimes\id}\circ\delta=\bar{\id\otimes\delta_G}\circ\delta,
\end{equation}
where $\delta_G$ is the canonical coaction on $C^*(G)$,
determined by $\bar{\delta_G}(x)=x\otimes x$ for $x\in G$
(and where $G$ is identified with its canonical image in $M(C^*(G))$).
Recall that $\delta$ gives rise to a 
right $B(G)$-module structure on $A^*$ given by
\[
\omega\cdot f=\bar{\omega\otimes f}\circ\delta
\quad\text{for $\omega\in A^*$ and $f\in B(G)$,}
\]
and also
to a left $B(G)$-module structure on $A$ given by
\[
f\cdot a=\bar{\id\otimes f}\circ\delta(a)\quad\text{for $f\in B(G)$ and $a\in A$},
\]
and that moreover
\[
(\omega\cdot f)(a)=\omega(f\cdot a)\quad\text{for all $\omega\in A^*$, $f\in B(G)$, and $a\in A$.}
\]

Further recall that
$1_G\cdot a=a$ for all $a\in A$, where $1_G$ is the constant function with value $1$.
In fact, suppose we have a homomorphism
$\delta:A\to M(A\otimes C^*(G))$ satisfying all the conditions of a coaction except perhaps injectivity. Then $\delta$ is in fact a coaction, because injectivity follows automatically,
by the following folklore trick:

\begin{lem}\label{injective}
Let $\delta:A\to M(A\otimes C^*(G))$ be a homomorphism satisfying 
\eqref{coaction nondegenerate} and 
\eqref{identity}.
Then for all $a\in A$ we have
\[
\bar{\id\otimes 1_G}\circ\delta(a)=a,
\]
where $1_G\in B(G)$ is the constant function with value $1$.
In particular, $\delta$ is injective and hence a coaction.
\end{lem}

\begin{proof}
First of all,
\begin{align*}
A
&=\clspn\Bigl\{(\id\otimes g)\bigl(\delta(a)(1\otimes c)\bigr):g\in B(G),a\in A,c\in C^*(G)\Bigr\}
\\&=\clspn\bigl\{\bar{\id\otimes c\cdot g}\circ\delta(a):g\in B(G),a\in A,c\in C^*(G)\bigr\}
\\&=\clspn\bigl\{\bar{\id\otimes f}\circ\delta(a):f\in B(G),a\in A\bigr\}.
\end{align*}
Now the following computation suffices: for all $a\in A$ and $f\in B(G)$ we have
\begin{align*}
&\bar{\id\otimes 1_G}\circ\delta\bigl(\bar{\id\otimes f}\circ\delta(a)\bigr)
\\&\quad=\bar{\id\otimes 1_G}\circ \bar{\id\otimes\id\otimes f}\circ (\delta\otimes\id)\circ\delta(a)
\\&\quad=\bar{\id\otimes 1_G\otimes f}\circ (\id\otimes\delta_G)\circ\delta(a)
\\&\quad=\bar{\id\otimes 1_Gf}\circ\delta(a)
\\&\quad=\bar{\id\otimes f}\circ\delta(a)
\qedhere
\end{align*}
\end{proof}

\section{Exotic quotients of $C^*(G)$}\label{certain quotients}

Let $G$ be a locally compact group,.
We are interested in certain quotients $C^*_E(G)$ (see \defnref{E quotient} for this notation).
We will always assume that ideals of $C^*$-algebras are closed and two-sided.
Let $B(G)$ denote the Fourier-Stieltjes algebra, which we identify with the dual of $C^*(G)$.
We give $B(G)$ the usual $C^*(G)$-bimodule structure: for $a,b\in C^*(G)$ and $f\in B(G)$ we define
\[
\<b,a\cdot f\>=\<ba,f\>\midtext{and}\<b,f\cdot a\>=\<ab,f\>.
\]
This bimodule structure extends  to an $M(C^*(G))$-bimodule structure, because for $m\in M(C^*(G))$ and $f\in B(G)$ the linear functionals $a\mapsto \<am,f\>$ and $a\mapsto \<ma,f\>$ on $C^*(G)$ are bounded.
Regarding $G$ as canonically embedded in $M(C^*(G))$, the associated $G$-bimodule structure on $B(G)$ is given by
\[
(x\cdot f)(y)=f(yx)\midtext{and}(f\cdot x)(y)=f(xy)
\]
for $x,y\in G$ and $f\in B(G)$.

A quotient $C^*(G)/I$ is uniquely determined by the annihilator $E=I\ann$ in $B(G)$,
which is a weak*-closed subspace.
We find it convenient to work in terms of $E$ rather than $I$, keeping in mind that we will have $I={}\ann E$, the preannihilator in $C^*(G)$.
First we record the following well-known property:

\begin{lem}
\label{invariant}
For any weak*-closed subspace $E$ of $B(G)$, the following are equivalent:
\begin{enumerate}
\item ${}\ann E$ is an ideal;

\item $E$ is a $C^*(G)$-subbimodule;

\item $E$ is $G$-invariant.
\end{enumerate}
\end{lem}

\begin{proof}
(1)$\Leftrightarrow$(2) follows from, e.g., \cite[Theorem~3.10.8]{ped}, and (2)$\Leftrightarrow$(3) follows by integration.
\end{proof}

\begin{defn}\label{E quotient}
If $E$ is a weak*-closed $G$-invariant subspace of $B(G)$, let $C^*_E(G)$ denote the quotient $C^*(G)/{}\ann E$.
\end{defn}

Note that the above definition makes sense, by \lemref{invariant}.

\begin{ex}
Of course we have
\[
C^*(G)=C^*_{B(G)}(G).
\]
Also,
\[
C^*_r(G)=C^*_{B_r(G)}(G),
\]
where $B_r(G)$ is the regular Fourier-Stieltjes algebra of $G$,
because if $\lambda:C^*(G)\to C^*_r(G)$ denotes the regular representation of $G$ then
\[
(\ker\lambda)\ann=B_r(G).
\]
Recall for later use that the intersection $C_c(G)\cap B(G)$ is norm-dense in the Fourier algebra $A(G)$ (for the norm of functionals on $C^*(G)$),
and is weak*-dense in $B_r(G)$ \cite{eym}.
\end{ex}

\begin{rem}
If $E$ is a weak*-closed $G$-invariant subspace of $B(G)$, and $q:C^*(G)\to C^*_E(G)$ is the quotient map, then the dual map $q^*:C^*_E(G)^*\to C^*(G)^*=B(G)$ is an isometric isomorphism onto $E$, and we identify $E=C^*_E(G)^*$ and regard $q^*$ as an inclusion map.
\end{rem}

Inspired in part by \cite{BrownGuentner}, we pause here to give another construction of the quotients $C^*_E(G)$:
\begin{enumerate}
\item
Start with a 
$G$-invariant, but \emph{not necessarily weak*-closed}, subspace $E$ of $B(G)$.

\item
Call a representation $U$ of $G$ on a Hilbert space $H$  an \emph{$E$-representation} if there is a dense subspace $H_0$ of $H$ such that the matrix coefficients
\[
x\mapsto \<U_x\xi,\eta\>
\]
are in $E$ for all $\xi,\eta\in H_0$.

\item
Define a $C^*$-seminorm $\|\cdot\|_E$ on $C_c(G)$ by
\[
\|f\|_E=\sup\{\|U(f)\|:\text{$U$ is an $E$-representation of $G$}\}.
\]
\end{enumerate}

The following lemma is presumably well-known, but we include a proof for the convenience of the reader.

\begin{lem}\label{DJ}
With the above notation, let $I$
be the  ideal of $C^*(G)$ 
given by
\begin{equation}
\label{kernel}
I=\{a\in C^*(G):\|a\|_E=0\}.
\end{equation}
Then:
\begin{enumerate}
\item$I={}\ann E$.

\item The weak*-closure $\bar E$ of $E$ in $B(G)$ is $G$-invariant,
and
$C^*_{\bar E}(G)=C^*(G)/I$ is the Hausdorff completion of $C_c(G)$ in the seminorm $\|\cdot\|_E$.

\item If $E$ is an ideal or a subalgebra of $B(G)$, then so is $\bar E$.
\end{enumerate}
\end{lem}

\begin{proof}
(1)
To show that $I\subset {}\ann E$, let $a\in I$
and  $f\in E$.
Since $f\in B(G)$, we can choose a representation $U$ of $G$ on a Hilbert space $H$ and vectors $\xi,\eta\in H$ such that
\[
f(x)=\<U_x\xi,\eta\>\quad\text{for $x\in G$.}
\]
Let $K_0$ be the smallest $G$-invariant subspace of $H$ containing both $\xi$ and $\eta$,
and let $K=\bar{K_0}$.
Then $K$ is a closed $G$-invariant subspace of $H$, so determines a subrepresentation $\rho$ of $G$.
For every $\zeta,\kappa\in K_0$, the function $x\mapsto \<U_x\zeta,\kappa\>$ is in $E$
because $E$
is $G$-invariant.
Thus $\rho$ is an $E$-representation.
We have
\begin{align*}
|\<a,f\>|
&=|\<\rho(a)\xi,\eta\>|
\\&\le \|\rho(a)\|\|\xi\|\|\eta\|
\\&\le \|a\|_E\|\xi\|\|\eta\|
\\&=0.
\end{align*}
Thus $a\in {}\ann E$.

For the opposite containment, suppose by way of contradiction that  we can find
$a\in {}\ann E\setminus I$.
Then $\|a\|_E\ne 0$, so we can also choose an $E$-representation $U$ of $G$ on a Hilbert space $H$ such that $U(a)\ne 0$.
Let $H_0$ be a dense subspace of $H$ such that for all $\xi,\eta\in H_0$ the function
$x\mapsto\<U_x\xi,\eta\>$
is in $E$.
By density we can choose $\xi,\eta\in H_0$ such that $\<U(a)\xi,\eta\>\ne 0$.
Then $g(x)=\<U_x\xi,\eta\>$ defines an element $g\in E$,
and we have
\[
\<a,g\>
=\<U(a)\xi,\eta\>
\ne 0,
\]
which is a contradiction.
Therefore ${}\ann E\subset I$,
as desired.

(2)
Since $I={}\ann E$ we have $\bar E=I\ann$, which is $G$-invariant because $I$ is an ideal, by \lemref{invariant}.
We have $I={}\ann \bar E$, so $C^*_{\bar E}(G)=C^*(G)/I$ by \defnref{E quotient}.
Since $C_c(G)$ is dense in $C^*(G)$, the result now follows by the definition of $I$ in \eqref{kernel}.

(3)
This follows immediately from separate weak*-continuity of multiplication in $B(G)$.
This is a well-known property of $B(G)$, but we include the brief proof here for completeness:
the bimodule action of $B(G)$ on the enveloping algebra $W^*(G)=B(G)^*$,
given by
\[
\<a\cdot f,g\>=\<a,fg\>=\<f\cdot a,g\>
\quad\text{for $a\in W^*(G),f,g\in B(G)$,}
\]
leaves $C^*(G)$ invariant, because it satisfies the submultiplicativity condition $\|a\cdot f\|\le \|a\|\|f\|$ on norms and leaves $C_c(G)\subset C^*(G)$ invariant.
Thus, if $f_i\to 0$ weak* in $B(G)$ and $g\in B(G)$, then for all $a\in C^*(G)$ we have
\[
\<a,f_ig\>=\<a\cdot g,f_i\>\to 0.
\qedhere
\]
\end{proof}

\begin{cor}\label{E rep}\

\begin{enumerate}
\item A representation $U$ of $G$ is an $E$-representation if and only if, identifying $U$ with the corresponding representation of $C^*(G)$, we have $\ker U\supset {}\ann E$.

\item A nondegenerate homomorphism $\tau:C^*(G)\to M(A)$, where $A$ is a $C^*$-algebra, factors through a homomorphism of $C^*_E(G)$ if and only if
\[
\bar\omega\circ\tau\in \bar E\quad\text{for all $\omega\in A^*$,}
\]
where again $\bar E$ denotes the weak*-closure of $E$.
\end{enumerate}
\end{cor}

\begin{proof}
This follows readily from \lemref{DJ}.
\end{proof}

\begin{rem}
In light of \lemref{DJ}, if we have a $G$-invariant subspace $E$ of $B(G)$ that is not necessarily weak*-closed, it makes sense to, and we shall, write $C^*_E(G)$ for $C^*_{\bar E}(G)$.
However, whenever convenient we can replace $E$ by its weak*-closure, giving the same quotient $C^*_E(G)$.
\end{rem}

\begin{obs}\label{immed}
By \lemref{DJ}, if $E$ is a $G$-invariant subspace of $B(G)$ then:
\begin{enumerate}
\item
$C^*_E(G)=C^*(G)$ if and only if $E$ is weak*-dense in $B(G)$.

\item
$C^*_E(G)=C^*_r(G)$ if and only if $E$ is weak*-dense in $B_r(G)$.
\end{enumerate}
\end{obs}

We record an elementary consequence of our definitions:

\begin{lem}\label{ccg}
For a weak*-closed $G$-invariant subspace $E$ of $B(G)$, the following are equivalent:
\begin{enumerate}
\item
${}\ann E\subset \ker\lambda$;

\item
$E\supset B_r(G)$;

\item
$E\supset A(G)$;

\item
$E\supset (C_c(G)\cap B(G))$;

\item
there is a \(unique\) homomorphism $\rho:C^*_E(G)\to C^*_r(G)$ 
making the diagram
\[
\xymatrix{
C^*(G) \ar[dr]^q \ar[dd]_\lambda
\\
&C^*_E(G) \ar@{-->}[dl]^\rho_{!}
\\
C^*_r(G)
}
\]
commute.
\end{enumerate}
\end{lem}

\begin{defn}\label{group algebra}
For a weak*-closed $G$-invariant subspace $E$ of $B(G)$, we say the quotient $C^*_E(G)$ is a \emph{\ga\ of $G$} if the above equivalent conditions (1)--(4) are satisfied.
If $B_r(G)\subsetneq E\ne B(G)$ we say the \ga\ is \emph{exotic}.
\end{defn}

We will see 
in \propref{graded}
that if $G$ is discrete then a quotient $C^*_E(G)$ is 
a \ga\ if and only if it is
topologically graded in Exel's sense \cite[Definition~3.4]{ExelAmenability}.

We are especially interested in \ga s that carry a coaction or a comultiplication.
We will need the following result, which is folklore among coaction cognoscenti:

\begin{lem}\label{quotient coaction}
If $\delta:A\to M(A\otimes C^*(G))$ 
is a coaction of $G$ on a $C^*$-algebra $A$ and $I$ is an ideal of $A$, then 
the following are equivalent:
\begin{enumerate}
\item
there is a coaction $\tilde\delta$ on $A/I$ making the diagram
\begin{equation}\label{quotient}
\xymatrix{
A \ar[r]^-\delta \ar[d]_q
&M(A\otimes C^*(G)) \ar[d]^{\bar{q\otimes\id}}
\\
A/I \ar[r]_-{\tilde\delta}
&M(A/I\otimes C^*(G))
}
\end{equation}
commute \(where $q$ is the quotient map\);

\item \label{kernel condition}
$I\subset \ker \bar{q\otimes \id}\circ \delta$.

\item $I\ann$ is a $B(G)$-submodule of $A^*$.
\end{enumerate}
\end{lem}

\begin{proof}
This is well-known, but difficult to find in the literature, so we include the brief proof for the convenience of the reader.
There exists a \emph{homomorphism} $\tilde\delta$ making the diagram~\eqref{quotient} commute if and only if 
(2)
holds,
and in that case $\tilde\delta$ will satisfy 
the coaction-nondegeneracy \eqref{coaction nondegenerate}
and the coaction identity \eqref{identity}.
By \lemref{injective} this implies that $\tilde\delta$ is a coaction.
Thus (1)$\Leftrightarrow$(2), and (2)$\Leftrightarrow$(3) follows from a routine calculation using the fact that
$\{\psi\otimes f:\psi\in (A/I)^*,f\in B(G)\}$ separates the elements of $M(A/I\otimes C^*(G))$.
\end{proof}

Recall that the multiplication in $B(G)$ satisfies
\[
\<a,fg\>=\<\delta_G(a),\bar{f\otimes g}\>\quad\text{for $a\in C^*(G)$ and $f,g\in B(G)$,}
\]
where here we use the notation $f\otimes g$ to denote the functional in $(C^*(G)\otimes C^*(G))^*$ determined by
\[
\<x\otimes y,\bar{f\otimes g}\>=f(x)g(y)\quad\text{for $x,y\in G$.}
\]

\begin{rem}\label{max}
Note that we need to explicitly state the above convention for $f\otimes g$, since we are using the minimal tensor product:
if $G$ is a group for which the canonical surjection
\[
C^*(G)\otimes_{\max} C^*(G)\to C^*(G)\otimes C^*(G)
\]
is noninjective\footnote{e.g., any infinite simple group with property T --- see \cite[Theorem~6.4.14 and Remark~6.4.15]{brownozawa}},
then
\begin{align*}
C^*(G)\otimes C^*(G)&\ne C^*(G\times G)
\\
(C^*(G)\otimes C^*(G))^*&\ne B(G\times G),
\end{align*}
because $C^*(G\times G)=C^*(G)\otimes_{\max} C^*(G)$.
\end{rem}

\begin{cor}\label{coaction}
Let $E$ be a weak*-closed $G$-invariant subspace of $B(G)$, and let $q:C^*(G)\to C^*_E(G)$ be the quotient map.
Then there is a coaction $\delta_G^E$ 
of $G$ on $C^*_E(G)$
such that
\[
\bar{\delta_G^E}(q(x))=q(x)\otimes x\quad\text{for $x\in G$}
\]
if and only if $E$ is an ideal of $B(G)$.
\end{cor}

\begin{proof}
Since $E$ is the annihilator of $\ker q$, this follows immediately from \lemref{quotient coaction}.
\end{proof}

Recall that in \defnref{group algebra} we called $C^*_E(G)$ a group $C^*$-algebra if $E$ is a weak*-closed $G$-invariant subspace of $B(G)$ containing $B_r(G)$;
this latter property is automatic if $E$ is an ideal (as long as it's nonzero):

\begin{lem}\label{smallest}
Every nonzero 
norm-closed $G$-invariant ideal of $B(G)$ contains $A(G)$, and hence
every nonzero weak*-closed $G$-invariant ideal of $B(G)$ contains $B_r(G)$.
\end{lem}

\begin{proof}
Let $E$ be the ideal.
It suffices to show that $E\cap A(G)$ is norm dense in $A(G)$.
There exist $t\in G$ and $f\in E$ such that $f(t)\ne 0$.
By \cite[Lemma~3.2]{eym} there exists 
$g\in A(G)\cap C_c(G)$ 
such that $g(t)\ne 0$, and then $fg\in E\cap C_c(G)$ is nonzero at $t$.
By $G$-invariance of $E$, for all $x\in G$ there exists $f\in E$ such that $f(x)\ne 0$.
Then for any $y\ne x$ we can find $g\in A(G)\cap C_c(G)$ such that $g(x)\ne 0$ and $g(y)=0$, and so $fg\in E$ is nonzero at $x$ and zero at $y$. Thus $E\cap A(G)$ is an ideal of $A(G)$ that is nowhere vanishing on $G$ and separates points, so by \cite[Corollary~3.38]{eym} $E\cap A(G)$ is norm dense in $A(G)$, so we are done.
\end{proof}

Recall that a \emph{comultiplication} on 
a  $C^*$-algebra $A$
is a  homomorphism (which we do \emph{not} in general require to be injective) $\Delta:A\to M(A\otimes A)$ satisfying the \emph{co-associativity} property
\[
\bar{\Delta\otimes\id}\circ\Delta=\bar{\id\otimes\Delta}\circ\Delta
\]
and the \emph{nondegeneracy properties}
\[
\clspn\{\Delta(A)(1\otimes A)\}=A\otimes A=\clspn\{(A\otimes 1)\Delta(A)\}.
\]
A $C^*$-algebra with a comultiplication is called a \emph{$C^*$-bialgebra} (see \cite{kawamura} for this terminology).
A comultiplication $\Delta$ on $A$ is used to make the dual space $A^*$ into a Banach algebra in the standard way:
\[
\omega\psi:=\bar{\omega\otimes\psi}\circ\Delta\quad\text{for $\omega,\psi\in A^*$.}
\]

The following is another folklore result, proved similarly to \lemref{quotient coaction}:

\begin{lem}\label{quotient comultiplication}
If $\Delta:A\to M(A\otimes A)$ 
is a comultiplication on a $C^*$-algebra $A$ and $I$ is an ideal of $A$, then 
the following are equivalent:
\begin{enumerate}
\item
there is a comultiplication $\tilde\Delta$ on $A/I$ making the diagram
\[
\xymatrix{
A \ar[r]^-\Delta \ar[d]_q
&M(A\otimes A) \ar[d]^{\bar{q\otimes q}}
\\
A/I \ar[r]_-{\tilde\Delta}
&M(A/I\otimes A/I)
}
\]
commute \(where $q$ is the quotient map\);

\item
$I\subset \ker \bar{q\otimes q}\circ \Delta$.

\item $I\ann$ is a subalgebra of $A^*$.
\end{enumerate}
\end{lem}

We apply this to
the canonical comultiplication $\delta_G$ on $C^*(G)$:

\begin{prop}\label{comultiplication}
Let $E$ be a weak*-closed $G$-invariant subspace of $B(G)$, and let $q:C^*(G)\to C^*_E(G)$ be the quotient map.
Then the following are equivalent:
\begin{enumerate}
\item
there is a comultiplication $\Delta$ making the diagram
\[
\xymatrix@C+30pt{
C^*(G) \ar[r]^-{\delta_G} \ar[d]_q
&M(C^*(G)\otimes C^*(G)) \ar[d]^{\bar{q\otimes q}}
\\
C^*_E(G) \ar[r]_-\Delta
&M(C^*_E(G)\otimes C^*_E(G))
}
\]
commute;

\item
${}\ann E\subset \ker\bar{q\otimes q}\circ\delta_G$;

\item
$E$ is a subalgebra of $B(G)$.
\end{enumerate}
\end{prop}

\begin{rem}\label{hopf}
\propref{comultiplication} tells us that if $E$ is a weak*-closed $G$-invariant subalgebra of $B(G)$, then
the group algebra $C^*_E(G)$ is a $C^*$-bialgebra.
However, this 
probably
does not make $C^*_E(G)$ a 
locally compact quantum group,
since this would require an antipode.
It might be difficult to investigate the general question of whether there exists \emph{some} antipode on $C^*_E(G)$ that is compatible with the comultiplication;
it seems more reasonable to ask whether the quotient map $q:C^*(G)\to C^*_E(G)$ takes the canonical antipode on $C^*(G)$ to an antipode on $C^*_E(G)$.
This requires $E$ to be closed under inverse 
i.e., if $f\in E$ then so is the function $f^\vee$ defined by $f^\vee(x)=f(x\inv)$.
Now, $f^\vee(x)=\bar{f^*(x)}$
where $f^*$ is defined by $f^*(a)=\bar{f(a^*)}$ for $a\in C^*(G)$.
Since $f\in E$ if and only if $f^*\in E$,
we see that $E$ is invariant under $f\mapsto f^\vee$ if and only if
it is invariant under complex conjugation.
In all our examples (in particular \secref{classical}) $E$ has this property.
Note that $C^*_E(G)$ always has a Haar weight, since we can compose the canonical Haar weight on $C^*_r(G)$ with the quotient map $C^*_E(G)\to C^*_r(G)$.
However, this Haar weight on $C^*_E(G)$ is faithful if and only if $E=B_r(G)$.
\end{rem}

\begin{rem}\label{closed}
By \lemref{DJ},
if $E$ is 
a $G$-invariant
ideal of $B(G)$
and $I={}\ann E$,
then
$\bar E$ is
also a $G$-invariant ideal,
so by \propref{coaction}
there is a coaction $\delta_G^E$ of $G$ on 
$C^*_E(G)$
such that
\[
\bar{\delta_G^E}(q(x))=q(x)\otimes x\quad\text{for $x\in G$,}
\]
where 
$q:C^*(G)\to C^*_E(G)$ is the quotient map.

Similarly, if $E$ is 
a $G$-invariant
subalgebra of $B(G)$
then
$\bar E$ is 
also a $G$-invariant subalgebra,
so by \propref{comultiplication}
there is a comultiplication $\Delta$ on 
$C^*_E(G)$
such that
\[
\bar{\Delta}(q(x))=q(x)\otimes q(x)\quad\text{for $x\in G$.}
\]
\end{rem}

\begin{ex}
Note that if the quotient $C^*_E(G)$ is a \ga, then the quotient map $q:C^*(G)\to C^*_E(G)$ is faithful on $C_c(G)$,
and so by \lemref{DJ} $C^*_E(G)$ is the completion of $C_c(G)$ in the associated norm $\|\cdot\|_E$.
However, 
$q$ being faithful on $C_c(G)$
is not sufficient for $C^*_E(G)$ to be a \ga.
The simplest example of this is in \cite[Exercise~XI.38]{FellDoran2} (which we modify only slightly): let $0\le a<b<2\pi$, and define a surjection
\[
q:C^*(\Z)\to C[a,b]
\]
by
\[
q(n)(t)=e^{int}.
\]
Then the unitaries $q(n)$ are linearly independent, 
so $q$ is faithful on $c_c(\Z)$,
but $q(C^*(\Z))$ is not a \ga\ because $\ker q$ is a nontrivial ideal of $C^*(\Z)$ and $\Z$ is amenable, so that $\ker\lambda=\{0\}$.
\end{ex}

\begin{ex}
The paper
\cite{EQInduced} shows how to construct exotic \ga s $C^*_E(G)$ 
(see also \cite[Remark~9.6]{Kyed} for similar exotic quantum groups)
with no coaction:
let
\[
q=\lambda\oplus 1_G,
\]
where $1_G$ denotes the trivial $1$-dimensional representation of $G$.
The quotient $C^*_E(G)$ is a \ga\ since $\ker q=\ker\lambda\cap \ker 1_G$.
On the other hand, we have
\[
E=(\ker q)\ann=B_r(G)+\C1_G,
\]
which is not an ideal of $B(G)$ unless it is all of $B(G)$, i.e., unless $q$ is faithful;
as remarked in \cite{EQInduced}, this behavior would be quite bizarre,
and in fact we do not know of any discrete nonamenable group with this property.

However, these quotients $C^*_E(G)$ are $C^*$-bialgebras, because $B_r(G)+\C 1_G$ is a subalgebra of $B(G)$.
Thus, these quotients give 
examples of exotic group $C^*$-bialgebras that are different from those in \cite[Proposition~4.4 and Remark~4.5]{BrownGuentner}. It is interesting to note that these quotients of $C^*(G)$ are of a decidedly elementary variety:
by \lemref{onto} we have
\[
C^*_E(G)=C^*_r(G)\oplus \C,
\]
because $C^*(G)=\ker\lambda+\ker 1_G$ since $G$ is nonamenable.
To see this latter implication, recall that if $G$ is nonamenable
then $1_G$ is not weakly contained in $\lambda$,
so $\ker 1_G\not\supset \ker\lambda$,
and hence $C^*(G)=\ker\lambda+\ker 1_G$ since $\ker 1_G$ is a maximal ideal.

Valette has a similar example in \cite[Theorem~3.6]{valetteT} where he shows that if $N$ is a closed normal subgroup of $G$ that has property (T), then 
$C^*(G)$ is the direct sum of $C^*(G/N)$ and a complementary ideal.

For a different source of exotic group $C^*$-bialgebras, see \exref{maxmin}.
\end{ex}

\begin{ex}
We can also find examples of \ga s with no comultiplication: modify the preceding example by taking
\[
q=\lambda\oplus \gamma,
\]
where $\gamma$ is a nontrivial character of $G$ (assuming that $G$ has such characters).
Then
\[
(\ker q)\ann=B_r(G)+\C \gamma,
\]
which is not a subalgebra of $B(G)$ when $G$ is nonamenable.
\end{ex}

\begin{ex}\label{maxmin}
Let $G$ be a locally compact group for which
the canonical surjection
\begin{equation}\label{max onto}
C^*(G)\otimes_{\max} C^*(G)\to C^*(G)\otimes C^*(G)
\end{equation}
is not injective,
where in the second tensor product we use the minimal $C^*$-tensor norm as usual (see \remref{max}).
Let $I$ denote the kernel of this map.
Since the algebraic product $B(G)\odot B(G)$ is weak*-dense in $(C^*(G)\otimes C^*(G))^*$, the annihilator $E=I\ann$ is the weak*-closed span of functions of the form
\[
(x,y)\mapsto f(x)g(y)\quad\text{for $f,g\in B(G)$.}
\]
This is clearly a subalgebra, but not an ideal, because it contains $1$.
Also, $E\supset B_r(G\times G)$ because the surjection \eqref{max onto} can be followed by
\[
C^*(G)\otimes C^*(G)\to C^*_r(G)\otimes C^*_r(G)\cong C^*_r(G\times G).
\]
Thus the canonical coaction $\delta_{G\times G}$ of $G\times G$ on $C^*(G\times G)$ descends to a comultiplication on the \ga\ $C^*_E(G\times G)\cong C^*(G)\otimes C^*(G)$, but not to a coaction of $G\times G$.
\end{ex}

\section{Classical ideals}\label{classical}

We continue to let $G$ be an arbitrary locally compact group.

We will apply the theory of the preceding sections to
\ga s $C^*_E(G)$ with $E$
of the form
\[
E=D\cap B(G),
\]
where $D$ is some familiar 
$G$-invariant 
set of functions on $G$.

\begin{notn}
If $D$ is a $G$-invariant set of functions on $G$, we write
$\|f\|_D=\|f\|_{D\cap B(G)}$, and similarly $C^*_D(G)=C^*_{D\cap B(G)}(G)$.
\end{notn}

So, for instance, we can consider $C^*_{C_c}(G)$, $C^*_{C_0(G)}(G)$, and $C^*_{L^p(G)}(G)$.
In each of these cases the intersection $E=D\cap B(G)$ is a $G$-invariant ideal of $B(G)$,
so by 
\remref{closed} and \lemref{smallest}
these quotients are all group $C^*$-algebras carrying coactions of $G$, and hence by \propref{comultiplication} they carry comultiplications. 
In the case that $G$ is discrete, $c_c(G)$, $c_0(G)$, and $\ell^p(G)$ could be regarded as classical ideals of $\ell^\infty(G)$;
this is the context of Brown and Guentner's ``new completions of discrete groups'' \cite{BrownGuentner}.

We have
\[
C^*_{C_c(G)}(G)=C^*_{A(G)}(G)=C^*_r(G),
\]
because $C_c(G)\cap B(G)$ is norm dense in $A(G)$, and hence weak*-dense in $B_r(G)$.
However, the  quotients $\csZg$ and $\csLpg$ are more mysterious.
Nevertheless, we have the following
(which, for the case of discrete $G$, is \cite[Proposition~2.11]{BrownGuentner}):

\begin{prop}\label{l2}
For all $p\le 2$ we have
$\csLpg=C^*_r(G)$.
\end{prop}

\begin{proof}
Since $L^p(G)\cap B(G)$ consists of bounded functions,
for $p\le 2$ we have
\[
C_c(G)\cap B(G)\subset L^p(G)\cap B(G)\subset L^2(G)\cap B(G).
\]
Now, 
if $U$ is a representation of $G$ having a cyclic vector $\xi$ such that the function $x\mapsto \<U_x\xi,\xi\>$ is in $L^2(G)$, then $U$ is contained in $\lambda$
(see, e.g., \cite{carey}),
and consequently $L^2(G)\cap B(G)\subset A(G)$.
Thus
\begin{align*}
B_r(G)
&=\wkstcl{C_c(G)\cap B(G)}
\\&\subset \wkstcl{L^p(G)\cap B(G)}
\\&\subset \wkstcl{L^2(G)\cap B(G)}
\\&\subset \wkstcl{A(G)}
\\&=B_r(G),
\end{align*}
and the result follows.
\end{proof}

\begin{rem}
\begin{enumerate}
\item
The proof of \propref{l2} is much easier when $G$ is discrete,
because
then for $\xi\in \ell^2(G)$ we have
\[
\xi(x)=\<\lambda_x\Chi_{\{e\}},\bar\xi\>,
\]
so $\ell^2(G)\subset A(G)$.

\item
In general, 
$\wkstcl{C_0(G)\cap B(G)}\supset B_r(G)$,
and the containment can be proper 
(for perhaps the earliest result along these lines, see \cite{menchoff}).
When $G$ is discrete, this phenomenon occurs precisely when $G$ is a-T-menable but nonamenable, by the result of \cite{BrownGuentner} mentioned in the introduction.

\item
Using the method outlined in this section, if we start with a $G$-invariant ideal $D$ of $L^\infty(G)$ and put $E=\wkstcl{D\cap B(G)}$, we get many weak*-closed ideals of $B(G)$, but probably not all.
For example, if we let $z_F$ be the supremum in the universal enveloping von Neumann algebra $W^*(G)=C^*(G)^{**}$ of the support projections of finite dimensional representations of $G$, then 
it follows from
\cite[Proposition~1, Theorem~2, Proposition~8]{walterstructure}
that $(1-z_F)\cdot B(G)$ is an ideal of $B(G)$
and $z_F\cdot B(G)=AP(G)\cap B(G)$ is a subalgebra.
It seems unlikely that for all locally compact groups $G$ the ideal $(1-z_F)\cdot B(G)$ arises as an intersection $D\cap B(G)$ for an ideal $D$ of $L^\infty(G)$.
\end{enumerate}
\end{rem}

\section{Graded algebras}\label{discrete}

In this
short
section we impose the condition that the group $G$ is discrete.
We made this a separate section for the purpose of clarity --- here the assumptions on $G$ are different from everywhere else in this paper.
\cite[Definition~3.1]{ExelAmenability} and \cite[VIII.16.11--12]{FellDoran2} define $G$-graded $C^*$-algebras as certain quotients of Fell-bundle algebras\footnote{\cite{ExelAmenability, FellDoran2} would require the images of the fibres to be linearly independent.}. When the fibres of the Fell bundle are $1$-dimensional, each one consists of scalar multiplies of a unitary. When these unitaries can be chosen to form a representation of $G$, 
the $C^*$-algebra is a quotient $C^*_E(G)$.

The following can be regarded as a special case of \cite[Theorem~3.3]{ExelAmenability}:

\begin{prop}\label{graded}
Let $E$ be a weak*-closed $G$-invariant subspace of $B(G)$,
and let $q:C^*(G)\to C^*_E(G)$ be the quotient map.
Then the following are equivalent:
\begin{enumerate}
\item
$C^*_E(G)$ is a \ga\ in the sense of \defnref{group algebra};

\item
there is a bounded linear functional $\omega$ on $C^*_E(G)$ such that
\[
\omega(q(x))=\begin{cases}
1\case x=e\\
0\case x\ne e;
\end{cases}
\]

\item
$E$ contains the canonical trace $\tr$ on $C^*(G)$;

\item
$E\supset B_r(G)$;

\item
there is a \(unique\) homomorphism $\rho:C^*_E(G)\to C^*_r(G)$ 
making the diagram
\[
\xymatrix{
C^*(G) \ar[dr]^q \ar[dd]_\lambda
\\
&C^*_E(G) \ar@{-->}[dl]^\rho_{!}
\\
C^*_r(G)
}
\]
commute.
\end{enumerate}
\end{prop}

\begin{proof}
Assuming (2),
the composition $\omega\circ q$ 
coincides with $\tr$,
so $\tr\in E$, and conversely if $\tr\in E$ then we get a suitable $\omega$. Thus (2) $\Leftrightarrow$ (3).

For the rest, just note that $B_r(G)=(\ker\lambda)\ann$ is the weak*-closed $G$-invariant subspace generated by $\tr=\Chi_{\{e\}}$, and appeal to \lemref{ccg}.
\end{proof}

\begin{rem}
Condition (2) in \propref{graded} is precisely what
Exel's
\cite[Definition~3.4]{ExelAmenability}
would require to say that $C^*_E(G)$ is \emph{topologically graded}.
\end{rem}

\section{Exotic coactions}\label{exotic coaction}

We return to the context of an arbitrary locally compact group $G$.

The  coactions appearing in noncommutative crossed-product duality come in a variety of flavors: \emph{reduced} vs. \emph{full} (see \cite[Appendix]{enchilada} or \cite{boiler}, for example), and, among the full ones, a spectrum with \emph{normal} and \emph{maximal} coactions at the extremes (see \cite{ekq}, for example).
In this concluding section we briefly propose a new program in crossed-product duality: ``exotic coactions'', involving the exotic \ga s $C^*_E(G)$ in the sense of \defnref{group algebra}.
From now until \propref{coaction bialgebra}
we are concerned with nonzero $G$-invariant weak*-closed ideals $E$ of $B(G)$.

By Lemmas~\ref{ccg} and \ref{smallest}
the quotient $C^*_E(G)=C^*(G)/{{}\ann E}$ is a \ga.
By \propref{coaction}, there is a coaction $\delta_G^E$ of $G$ on $C^*_E(G)$ 
making the diagram
\[
\xymatrix{
C^*(G) \ar[r]^-{\delta_G} \ar[d]_q
&M(C^*(G)\otimes C^*(G)) \ar[d]^{\bar{q\otimes\id}}
\\
C^*_E(G) \ar[r]_-{\delta_G^E}
&M(C^*_E(G)\otimes C^*(G))
}
\]
commute, where $q$ is the quotient map,
and by \propref{comultiplication} there is a quotient comultiplication $\Delta$ on $C^*_E(G)$.
Recall that we defined the \emph{exotic}
\ga s
to be the ones strictly between the two extremes $C^*(G)$ and $C^*_r(G)$, corresponding to $E=B(G)$ and $E=B_r(G)$, respectively.

On one level, we could try to study coactions of Hopf $C^*$-algebras associated to the locally compact group $G$ other than $C^*(G)$ and $C^*_r(G)$.
But there is an inconvenient subtlety here (see \remref{hopf}).
However, there is a deeper level to this program, relating more directly to crossed-product duality. At the deepest level, we aim for a characterization of \emph{all} coactions of $G$ in terms of the quotients $C^*_E(G)$. 
We hasten to emphasize that at this time 
some of the following is speculative, and is intended merely to outline a program of study.

From now on, the unadorned term ``coaction'' will refer to a full coaction of $G$ on a $C^*$-algebra $A$.

Let $\psi:(A^m,\delta^m)\to (A,\delta)$ be the maximalization of $\delta$,
so that $\delta^m$ is a maximal coaction, $\psi:A^m\to A$ is an equivariant surjection, and the crossed-product surjection
\[
\psi\times G:A^m\rtimes_{\delta^m} G\to A\rtimes_\delta G
\]
(for the existence of which, see \cite[Lemma~A.46]{enchilada}, for example)
is an isomorphism.
Since $\delta^m$ is maximal, the canonical surjection
\[
\Phi:A^m\rtimes_{\delta^m} G\rtimes_{\widehat{\delta^m}} G\to A^m\otimes\KK(L^2(G))
\]
is an isomorphism (this is ``full-crossed-product duality'').
Blurring the distinction between $A^m\rtimes_{\delta^m} G$ and the isomorphic crossed product $A\rtimes_\delta G$, 
and recalling that $\psi\times G:A^m\rtimes_{\delta^m} G\to A\rtimes_\delta G$ is $\widehat{\delta^m}-\widehat\delta$ equivariant,
we can regard $\Phi$ as an isomorphism
\[
\xymatrix{
A\rtimes_{\delta} G\rtimes_{\widehat\delta} G \ar[r]^-\Phi_-\cong
&A^m\otimes\KK(L^2(G)).
}
\]
We have a surjection
\[
\psi\otimes\id:A^m\otimes\KK(L^2(G))\to A\otimes\KK(L^2(G)),
\]
whose kernel is $(\ker\psi)\otimes\KK(L^2(G))$ since $\KK(L^2(G))$ is nuclear.
Let $K_\delta$ be the inverse image under $\Phi$ of this kernel, giving an ideal of $A\rtimes_{\delta} G\rtimes_{\widehat\delta} G$ and an isomorphism $\Phi_\delta$ making the diagram
\begin{equation}\label{Q}
\xymatrix@C+30pt{
A\rtimes_\delta G\rtimes_{\widehat\delta} G
\ar[r]^-\Phi_-\cong \ar[d]_Q
&A^m\otimes \KK(L^2(G))
\ar[d]^{\psi\otimes\id}
\\
(A\rtimes_\delta G\rtimes_{\widehat\delta} G)/K_\delta
\ar[r]_-{\Phi_\delta}^-\cong
&A\otimes \KK(L^2(G))
}
\end{equation}
commute, where $Q$ is the quotient map.
Adapting the techniques of \cite[Theorem~3.7]{eq:full}\footnote{This is a convenient place to correct a slip in the last paragraph of the proof of \cite[Theorem~3.7]{eq:full}: ``contains'' should be replaced by ``is contained in'' (both times).},
it is not hard to see that $K_\delta$ is contained in the kernel of the regular representation 
$\Lambda:A\rtimes_\delta G\rtimes_{\widehat\delta} G\to A\rtimes_\delta G\rtimes_{\widehat\delta,r} G$.

If $\delta$ is maximal, then 
diagram~\ref{Q} collapses to a single row.
On the other hand, if $\delta$ is normal, then $Q$ is the regular representation $\Lambda$ and in particular
\[
(A\rtimes_\delta G\rtimes_{\widehat\delta} G)/K_\delta=A\rtimes_\delta G\rtimes_{\widehat\delta,r} G.
\]
(In this case the isomorphism $\Phi_\delta$ is ``reduced-crossed-product duality''.)

With the ultimate goal 
(which at this time remains elusive
--- see Conjectures~\ref{E-coaction} and \ref{E dual})
of achieving an ``$E$-crossed-product duality'',
intermediate between full- and reduced-crossed-product dualities,
below we will propose tentative definitions of 
``$E$-crossed-product duality''
and ``$E$-crossed products'' $B\rtimes_{\alpha,E} G$ by actions $\alpha:G\to \aut B$,
and we will prove that they have the following properties:
\begin{enumerate}
\item a coaction satisfies $B(G)$-\duality\ if and only if it is maximal.
\item a coaction satisfies $B_r(G)$-\duality\ if and only if it is normal.
\item $B\rtimes_{\alpha,B(G)} G=B\rtimes_\alpha G$.
\item $B\rtimes_{\alpha,B_r(G)} G=B\rtimes_{\alpha,r} G$.
\item The dual coaction $\hat\alpha$ on the full crossed product $B\rtimes_\alpha G$ 
satisfies $B(G)$\duality.
\item The dual coaction $\hat\alpha^n$ on the reduced crossed product $B\rtimes_{\alpha,r} G$ 
satisfies $B_r(G)$\duality.
\item In general, $B\rtimes_{\alpha,E} G$ is a quotient of $B\rtimes_\alpha G$ by an ideal contained in the kernel of the regular representation
\[
\Lambda:B\rtimes_\alpha G\to B\rtimes_{\alpha,r} G.
\]
\item There is a dual coaction $\hat\alpha_E$ of $G$ on $B\times_{\alpha,E} G$.
\end{enumerate}

\begin{defn}\label{J_E}
Define an ideal $J_{\alpha,E}$ of the crossed product $B\rtimes_\alpha G$ by
\[
J_{\alpha,E}=\ker \bar{\id\otimes q}\circ \hat\alpha,
\]
and define the \emph{$E$-crossed product} by
\[
B\rtimes_{\alpha,E} G=(B\rtimes_\alpha G)/J_{\alpha,E}.
\]
\end{defn}

Note that the above properties (1)--(7) are obviously satisfied (because $\hat\alpha$ is maximal and $\hat\alpha^n$ is normal), and we now verify that (8) holds as well:

\begin{thm}\label{dual}
Let $E$ be a nonzero weak*-closed $G$-invariant ideal of $B(G)$, and
let $Q:B\rtimes_\alpha G\to B\rtimes_{\alpha,E} G$ be the quotient map.
Then there is a coaction $\hat\alpha_E$ 
making the diagram
\[
\xymatrix{
B\rtimes_\alpha G \ar[r]^-{\hat\alpha} \ar[d]_{Q}
&M((B\rtimes_\alpha G)\otimes C^*(G)) \ar[d]^{\bar{Q\otimes\id}}
\\
B\rtimes_{\alpha,E} G \ar[r]_-{\hat\alpha_E}
&M((B\rtimes_{\alpha,E} G)\otimes C^*(G))
}
\]
commute.
\end{thm}

\begin{proof}
By \lemref{coaction}, we must show that
\[
J_{\alpha,E}\subset \ker \bar{Q\otimes\id}\circ \hat\alpha.
\]
Let $a\in J_{\alpha,E}$, $\omega\in (B\rtimes_{\alpha,E} G)^*$, and $g\in B(G)$.
Then
\begin{align*}
\bar{\omega\otimes g}\circ \bar{Q\otimes\id}\circ\hat\alpha(a)
&=\bar{Q^*\omega\otimes g}\circ\hat\alpha(a)
\\&=Q^*\omega\circ \bar{\id\otimes g}\circ\hat\alpha(a)
\\&=Q^*\omega(g\cdot a).
\end{align*}
Now, since $Q^*\omega\in J_{\alpha,E}\ann$, it suffices to show that $g\cdot a\in J_{\alpha,E}$.
For $h\in E$ we have
\[
h\cdot (g\cdot a)=(hg)\cdot a=(gh)\cdot a=g\cdot (h\cdot a)=0,
\]
because $h\cdot a=0$ by \lemref{kill} below.
\end{proof}

\begin{lem}\label{kill}
With the above notation, we have:
\begin{enumerate}
\item $J_{\alpha,E}=\{a\in B\rtimes_\alpha G:E\cdot a=\{0\}\}$, and
\item $J_{\alpha,E}\ann=\clspn\{(B\rtimes_\alpha G)^*\cdot E\}$,
where the closure is in the weak*-topology.
\end{enumerate}
\end{lem}

\begin{proof}
(1)
For $a\in B\rtimes_\alpha G$, we have
\begin{align*}
&a\in J_{\alpha,E}
\\&\quad\Leftrightarrow \bar{\id\otimes q}\circ\hat\alpha(a)=0
\\&\quad\Leftrightarrow \bar{\omega\otimes h}\circ\bar{\id\otimes q}\circ\hat\alpha(a)=0
\\&\hspace{1in}\text{for all $\omega\in (B\rtimes_{\alpha,E} G)^*$ and $h\in C^*_E(G)^*$}
\\&\quad\Leftrightarrow \bar{\omega\otimes q^*h}\circ\hat\alpha(a)=0
\\&\hspace{1in}\text{for all $\omega\in (B\rtimes_{\alpha,E} G)^*$ and $h\in C^*_E(G)^*$}
\\&\quad\Leftrightarrow \bar{\omega\otimes g}\circ\hat\alpha(a)=0
\\&\hspace{1in}\text{for all $\omega\in (B\rtimes_{\alpha,E} G)^*$ and $g\in E$}
\\&\quad\Leftrightarrow \bar\omega\circ\bar{\id\otimes g}\circ\hat\alpha(a)=0
\\&\hspace{1in}\text{for all $\omega\in (B\rtimes_{\alpha,E} G)^*$ and $g\in E$}
\\&\quad\Leftrightarrow \omega(g\cdot a)=0
\quad\text{for all $\omega\in (B\rtimes_{\alpha,E} G)^*$ and $g\in E$}
\\&\quad\Leftrightarrow g\cdot a=0
\quad\text{for all $g\in E$.}
\end{align*}

(2)
If $a\in J_{\alpha,E}$, $\omega\in (B\rtimes_\alpha G)^*$, and $f\in E$,
\[
(\omega\cdot f)(a)=\omega(f\cdot a)=0,
\]
so $\omega\cdot f\in J_{\alpha,E}\ann$, and hence the left-hand side contains the right.

For the opposite containment, it suffices to show that
\[
J_{\alpha,E}\supset {}\ann\bigl((B\rtimes_\alpha G)^*\cdot E\bigr).
\]
If $a\in {}\ann((B\rtimes_\alpha G)^*\cdot E)$, then for all $\omega\in (B\rtimes_\alpha G)^*$ and $f\in E$ we have
\begin{align*}
0
&=(\omega\cdot f)(a)
=\omega(f\cdot a),
\end{align*}
so $f\cdot a=0$, and therefore $a\in J_{\alpha,E}$.
\end{proof}

\begin{rem}
We could define a covariant representation $(\pi,U)$ of the action $(B,\alpha)$ to be an \emph{$E$-representation} if the representation $U$ of $G$ is an $E$-representation,
and we could define an ideal $\tilde J_{\alpha,E}$ of $B\rtimes_\alpha G$ by
\begin{equation}\label{tilde ideal}
\tilde J_{\alpha,E}=\{a:\pi\times U(a)=0\text{ for every $E$-representation $(\pi,U)$}\},
\end{equation}
similarly to what is done in \cite[Definition~5.2]{BrownGuentner}.
It follows from \corref{E rep} that $(\pi,U)$ is an $E$-representation in the above sense if and only if
\[
\bar\omega\circ U\in E\quad\text{for all $\omega\in \bigl(\pi\times U(B\rtimes_\alpha G)\bigr)^*$,}
\]
where $i_G:C^*(G)\to M(B\rtimes_\alpha G)$ is the canonical nondegenerate homomorphism,
and consequently
\[
\tilde J_{\alpha,E}\ann=\{\omega\in (B\rtimes_\alpha G)^*:\bar\omega\circ i_G\in E\}.
\]
In the following lemma we show one
containment that always holds between \eqref{tilde ideal} and the ideal of \defnref{J_E},
after which we explain why
these ideals do \emph{not} coincide in general.
\end{rem}

\begin{lem}\label{ideals}
With the above notation, we have
\[
\tilde J_{\alpha,E}\subset J_{\alpha,E}.
\]
\end{lem}

\begin{proof}
If $\omega\in (B\rtimes_\alpha G)^*$ and $f\in E$, then
\begin{align*}
\bar{\omega\cdot f}\circ i_G
&=\bar{\omega\otimes f}\circ\bar{\hat\alpha}\circ i_G
\\&=\bar{\omega\otimes f}\circ\bar{i_G\otimes\id}\circ \delta_G
\\&=\bar{\bar\omega\circ i_G\otimes f}\circ \delta_G
\\&=\bigl(\bar\omega\circ i_G\bigr)f,
\end{align*}
which is in $E$ because $f\in E$ and $E$ is an ideal of $B(G)$.
Thus $\omega\cdot f\in \tilde J_{\alpha,E}\ann$.
\end{proof}

\begin{ex}
To see that the inclusion of \lemref{ideals} can be proper, 
consider the extreme case $E=B_r(G)$, so that $B\rtimes_{\alpha,E} G=B\rtimes_{\alpha,r} G$.
In this case $J_{\alpha,E}$ is the kernel of the regular representation $\Lambda:B\rtimes_\alpha G\to B\rtimes_{\alpha,r} G$.
On the other hand, $\tilde J_{\alpha,E}$ comprises the elements that are killed by every representation $\pi\times U$ for which $U$ is weakly contained in the regular representation $\lambda$ of $G$.
\cite[Example~5.3]{QS} gives an example of an action $(B,\alpha)$
having a covariant representation $(\pi,U)$ for which $U$ is weakly contained in $\lambda$ but $\pi\times U$ is not weakly contained in $\Lambda$.
Thus $\ker \pi\times U$ contains $\tilde J_{\alpha,E}$ and $J_{\alpha,E}$ has an element not contained in $\ker \pi\times U$, so $\tilde J_{\alpha,E}$ is properly contained in $J_{\alpha,E}$ in this case.
\end{ex}

\begin{defn}
We say that $G$ is \emph{$E$-amenable} if there are positive definite functions $h_n$ in $E$ such that $h_n\to 1$ uniformly on compact sets.
\end{defn}

\begin{lem}
If $G$ is $E$-amenable and $(A,G,\alpha)$ is an action, then $J_{\alpha,E}=\{0\}$,
so
\[
A\rtimes_\alpha G\cong A\rtimes_{\alpha,E} G.
\]
\end{lem}

\begin{proof}
By \lemref{kill}, we have $h_n\cdot a=0$ for all $a\in J_{\alpha,E}$.
Since $h_n\to 1$ uniformly on compact sets,
it follows that $h_n\cdot a\to a$ in norm.
To see this, note that since the $h_n$ are positive definite and $h_n\to 1$, the sequence $\{h_n\}$ is bounded in $B(G)$, and certainly for $f\in C_c(G)$ we have
\[
h_n\cdot \bigl(fa\bigr)=(h_n f)a\to fa
\]
in norm, because the pointwise products $h_n f$ converge to $f$ uniformly and hence in the inductive limit topology since $\supp f$ is compact.
Therefore $J_{\alpha,E}=\{0\}$.
\end{proof}

\begin{rem}\label{amenable}
In \cite[Section~5]{BrownGuentner}, Brown and Guentner study actions of a discrete group $G$ on a unital abelian $C^*$-algebra $C(X)$,
and introduce the concept of a $D$-amenable action, where $D$ is a $G$-invariant ideal of $\ell^\infty(G)$.
In particular, if $G$ is $D$-amenable then every action of $G$ is $D$-amenable.
They show that if the action is $D$-amenable then
$\tilde J_{\alpha,E}=\{0\}$,
i.e.,
\[
C^*_D(X\rtimes G)\cong C(X)\rtimes_\alpha G.
\]
Here we have used the notation of \cite{BrownGuentner}:
$C^*_D(X\rtimes G)$ denotes the quotient of the crossed product $C(X)\rtimes_\alpha G$ by the ideal $\tilde J_{\alpha,E}$ (although Brown and Guentner give a different, albeit equivalent, definition).

\begin{q}
With the above notation,
form a weak*-closed $G$-invariant ideal $E$ of $B(G)$ by taking the weak*-closure of $D\cap B(G)$.
Then is the stronger statement $J_{\alpha,E}=\{0\}$ true?
(One easily checks it for $E=B_r(G)$,
and it is trivial for $E=B(G)$.)
\end{q}

Note that the techniques of \cite{BrownGuentner} rely heavily on the fact that they are using ideals of $\ell^\infty(G)$, whereas our methods require ideals of $B(G)$.
\end{rem}

\begin{defn}
A coaction $(A,\delta)$ \emph{\satisfy} if
\[
K_\delta=J_{\widehat\delta,E},
\]
where $K_\delta$ is the ideal from \eqref{Q} and $J_{\widehat\delta,E}$ is the ideal associated to the dual action $\widehat\delta$ in \defnref{J_E}.
\end{defn}

Thus $(A,\delta)$ \satisfy\ precisely when we have an isomorphism
$\Phi_E$ making the diagram
\[
\xymatrix{
A\rtimes_\delta G\rtimes_{\what\delta} G \ar[r]^-\Phi \ar[d]_Q
&A\otimes \KK(L^2(G))
\\
A\rtimes_\delta G\rtimes_{\what\delta,E} G \ar[ur]_{\Phi_E}^\cong
}
\]
commute,
where $Q$ is the quotient map.

\begin{conj}\label{E-coaction}
Every coaction \satisfy\ for some $E$.
\end{conj}

\begin{obs}\label{trivial}
If $E$ is an ideal of $B(G)$, then every
\ga\ $C^*_E(G)$ is an $E$-crossed product:

\[
C^*_E(G)=\C\rtimes_{\iota,E} G,
\]
where $\iota$ is the trivial action of $G$ on $\C$, 
because the kernel of the quotient map $C^*(G)\to C^*_E(G)$ is ${}\ann E$.
This generalizes the extreme cases
\begin{enumerate}
\item $C^*(G)=\C\rtimes_\iota G$;
\item $C^*_r(G)=\C\rtimes_{\iota,r} G$.
\end{enumerate}
\end{obs}

\begin{conj}\label{E dual}
If $(B,\alpha)$ is an action, then the dual coaction $\hat\alpha_E$ on the $E$-crossed product $B\rtimes_{\alpha,E} G$ \satisfy.
\end{conj}

\begin{rem}
In particular, by \obsref{trivial}, 
\conjref{E dual} would imply as a special case
that the canonical coaction $\delta_G^E$ on the group algebra $C^*_E(G)$ \satisfy.
\end{rem}

For our final result, we
only require that $E$ be 
a weak*-closed $G$-invariant subalgebra of $B(G)$ (but not necessarily an ideal).
By \propref{comultiplication}, $C^*_E(G)$ carries a comultiplication $\Delta$ that is a quotient of the canonical comultiplication $\delta_G$ on $C^*(G)$.

Techniques similar to those used in the proof of \thmref{dual},
taking $g\in E$ rather than $g\in B(G)$,
can be used to show:

\begin{prop}\label{coaction bialgebra}
Let $E$ be a weak*-closed $G$-invariant subalgebra of $B(G)$,
and let $(B,\alpha)$ be an action.
Then there is a coaction $\Delta_\alpha$ of the $C^*$-bialgebra $C^*_E(G)$ making the diagram
\[
\xymatrix{
B\rtimes_\alpha G \ar[r]^-{\hat\alpha} \ar[d]_{Q}
&M((B\rtimes_\alpha G)\otimes C^*(G)) \ar[d]^{\bar{Q\otimes q}}
\\
B\rtimes_{\alpha,E} G \ar[r]_-{\Delta_\alpha}
&M((B\rtimes_{\alpha,E} G)\otimes C^*_E(G))
}
\]
commute, where we use notation from \thmref{dual}.
\end{prop}

We close with a rather vague query:

\begin{q}
What are the relationships among $E$-crossed products, $E$-coactions, and coactions of the $C^*$-bialgebra $C^*_E(G)$?
\end{q}

We hope to investigate this question, together with Conjectures~\ref{E-coaction} and \ref{E dual}, in future research.


\newcommand{\etalchar}[1]{$^{#1}$}
\providecommand{\bysame}{\leavevmode\hbox to3em{\hrulefill}\thinspace}
\providecommand{\MR}{\relax\ifhmode\unskip\space\fi MR }
\providecommand{\MRhref}[2]{%
  \href{http://www.ams.org/mathscinet-getitem?mr=#1}{#2}
}
\providecommand{\href}[2]{#2}

\end{document}